\documentclass[11pt,leqno]{article}
\usepackage{amssymb,amsmath,amsthm,amsxtra}
\usepackage[margin=3cm]{geometry} 
\usepackage{hyperref} 
\usepackage[all]{xy} 

\swapnumbers 

\newtheorem{theorem}[equation]{Theorem}
\newtheorem{proposition}[equation]{Proposition} 

\newtheorem{corollary}[equation]{Corollary}

\theoremstyle{definition}
\newtheorem{example}[equation]{Example}

\newtheorem{para}[equation]{} 

\newtheorem{remark}[equation]{Remark}

\makeatletter 
\@addtoreset{equation}{section} 
\makeatother

\renewcommand{\Re}{\operatorname{Re}}
\renewcommand{\Im}{\operatorname{Im}} 

\newcommand{\Z}{\mathbb{Z}}

\newcommand{\R}{\mathbb{R}} 
\newcommand{\C}{\mathbb{C}} 
\newcommand{\PP}{\mathbb{P}}
\newcommand{\JJ}{\mathbb{J}} 

\newcommand{\hGamma}{\widehat{\Gamma}}

\newcommand{\cO}{\mathcal{O}} 
\newcommand{\cI}{\mathcal{I}} 
\newcommand{\hcI}{\widehat{\cI}} 
\newcommand{\tI}{\widetilde{I}}
\newcommand{\cM}{\mathcal{M}} 
\newcommand{\cV}{\mathcal{V}}

\newcommand{\ch}{\operatorname{ch}} 
\newcommand{\End}{\operatorname{End}}

\newcommand{\vol}{\operatorname{vol}} 
\newcommand{\Ch}{\operatorname{Ch}}

\newcommand{\iu}{\mathtt{i}}

\def\corr#1{\left\langle #1 \right\rangle}

\begin{document} 

\title{Mirror symmetric Gamma conjecture for Fano and Calabi-Yau manifolds}  
\author{Hiroshi Iritani} 
\date{\empty}

\maketitle 

\begin{abstract} 
The mirror symmetric Gamma conjecture roughly speaking says that 
the Gamma class of a manifold determines the asymptotics of (exponential) periods of the mirror.  
We recast the method in \cite{Iritani:periods} in a more general context 
and show that the mirror symmetric Gamma conjecture for a Fano manifold $F$ implies, 
via Laplace transformation, that for the total space $K_F$ of the canonical bundle 
or for anticanonical sections in $F$. 
More generally, we discuss the mirror symmetric Gamma conjecture for 
the total space of a sum of anti-nef line bundles over $F$ 
or for nef complete intersections in $F$. 
\end{abstract} 


\section{Overview} 
In this paper we compare the `mirror symmetric' 
Gamma conjectures for Fano and Calabi-Yau manifolds, which roughly speaking say the following. 
\begin{itemize} 
\item[(1)] For a Fano manifold $F$ and its Laurent polynomial mirror $W(x)$, we have 
\[
\int_{(\R_{>0})^n} e^{-W(x)/z} \frac{dx_1\cdots dx_n}{x_1\cdots x_n}  
\sim \int_F z^{c_1} \hGamma_F \quad \text{as $z\to +\infty$} 
\]
with $\hGamma_F$ the Gamma class of $F$. 
\item[(2)] For a Calabi-Yau manifold $Y$ with K\"ahler form in the class $\tau\in H^2(Y)$ and its mirror Calabi-Yau manifold $Z_\tau$ with a holomorphic volume form $\Omega_\tau$, we have 
\[
\int_{C_\tau} \Omega_\tau \sim \int_Y e^{-\tau} \hGamma_Y \quad 
\text{as $\tau$ approaches the large radius limit} 
\]
for some integral cycles $C_\tau$.  
\end{itemize} 
Here the \emph{Gamma class} of a complex manifold is the characteristic class of the tangent bundle associated with the Taylor expansion of the Euler $\Gamma$-function $\Gamma(z)$ at $z=1$ and it can be expressed in terms of the Chern classes and Riemann zeta values (see \eqref{eq:Gamma_class}).   
For a Fano manifold $F$, we can construct a Calabi-Yau manifold by either taking the total space $K_F$ of the canonical bundle or taking an anticanonical section $Y\subset F$. 
We study the relationship between the mirror symmetric Gamma conjectures in these cases. 

More precisely, if a Laurent polynomial mirror $W$ of $F$ satisfies the mirror symmetric Gamma conjecture, we show that a certain relative period of the pair $((\C^\times)^n, W^{-1}(-1))$ satisfies the Gamma conjecture for the local Calabi-Yau $K_F$ and a period of the fiber $W^{-1}(1)$ satisfies the Gamma conjecture for an anticanonical Calabi-Yau $Y\subset F$. 
The main results will be stated in Theorems \ref{thm:main} and \ref{thm:generalizations} below. 

\section{Non-mirror-symmetric Gamma conjecture for Fano manifolds in a nutshell} 
\label{sec:nutshell} 
Non-mirror-symmetric Gamma conjecture for Fano manifolds \cite{GGI} 
roughly speaking says that one can extract topological 
information (the Gamma class) of a Fano manifold $F$ by counting rational curves in $F$. 
Recall that the Gamma class $\hGamma_F$ is a characteristic class defined by 
\begin{align}
\label{eq:Gamma_class} 
\begin{split}  
\hGamma_F & = \Gamma(1+\delta_1) \cdots 
\Gamma(1+\delta_n) \\
& = \exp \left(-\gamma c_1 + \sum_{k=2}^\infty 
(-1)^k \zeta(k) (k-1)! \ch_k(TF)\right) 
\end{split} 
\end{align} 
where $c_1 = c_1(TF)$, $\delta_1,\dots,\delta_n$ are the Chern roots of 
$TF$ so that $c(TF) = (1+\delta_1) \cdots (1+\delta_n)$, 
$\gamma=\lim_{n\to\infty} (\sum_{k=1}^n \frac{1}{k} - \log n)$ 
is the Euler constant and $\zeta(k)$ is the Riemann zeta function. 
Let $J_F(\tau,z)$ denote the Givental $J$-function 
\cite{Givental:symplectic} of $F$: 
\[
J_F(\tau,z) = e^{\tau/z}\left(1+ \sum_{d\in H_2(F,\Z), d \neq 0}
e^{\tau \cdot d}
\sum_i  
\corr{\frac{\phi^i}{z(z-\psi)}}_{0,1,d} \phi_i  \right) 
\]
where $\tau \in H^2(F)$ and $\{\phi_i\}$ and $\{\phi^i\}$ 
are dual bases of $H^*(F)$ 
such that $\int_F \phi_i \cup \phi^j 
= \delta_i^j$. This is a generating function of genus-zero 
one-point descendant Gromov-Witten invariants. 
The \emph{Gamma conjecture I} \cite[\S 3]{GGI} says that 
we should have 
\begin{equation} 
\label{eq:Gamma-I}
[\hGamma_F] = \lim_{t\to +\infty} [J_F(c_1 \log t,1)] 
\end{equation} 
in the \emph{projective space} $\PP(H^*(F))$ of cohomology, 
where the square bracket $[\cdots]$ means a point in the projective space. 
Note that the right-hand side is determined only by counting rational 
curves (Gromov-Witten invariants) in $F$ whereas the left-hand 
side contains topological information of $F$. 

\para 
\label{para:J_DE} The $J$-function $J_F(c_1\log t,1)$ 
is a solution to the quantum differential equation 
which has irregular singularities at $t=\infty$; 
the above limit \eqref{eq:Gamma-I} detects the most dominant 
component of the 
solution $J_F$ as $t\to +\infty$. 
When the Gamma conjecture I holds, 
the Gamma class also arises as a row of the connection matrix 
between fundamental solutions around 
regular and irregular singularities of the quantum connection. 

\begin{remark} 
The (inverse) Gamma class was introduced by Libgober \cite{Libgober} motivated by an observation in mirror symmetry due to Hosono, Klemm, Theisen and Yau \cite{HKTY}. 
The $\hGamma$-integral/rational structure in quantum cohomology has been introduced by the author \cite{Iritani:Int} and Katzarkov, Kontsevich and Pantev \cite{KKP}. 
The Gamma conjecture and related conjectures have been proved in some cases. 
The Gamma conjecture I has been proved for Fano toric manifolds/toric complete intersections \cite{Iritani:Int, Iritani:periods, GI} (assuming a version of Conjecture $\cO$), Fano threefolds of Picard rank one \cite{Golyshev-Zagier}, 
Grassmannians of type A \cite{Golyshev:deresonating, GGI} and del Pezzo surfaces \cite{HKCY}. 
The Conjecture $\cO$ (which is closely related to the Gamma conjecture I) has been proved for homogeneous varieties $G/P$ \cite{Cheong-Li}. 
The Gamma conjecture II (which we do not discuss in this article but is closely related to Dubrovin's conjecture \cite{Dubrovin:ICM}) has been proved for Grassmannians of type A \cite{GGI}, Fano toric manifolds \cite{Iritani:Int, Fang-Zhou:GammaII_toric} and quadric hypersurfaces \cite{Hu-Ke:GammaII_quadric}. 
See also \cite{AvSZ,  Galkin:Apery, Golyshev:deresonating} for the Ap\'ery limits which are closely related to the Gamma conjecture I. 
\end{remark} 

%

\section{Mirror symmetric Gamma conjecture for Fano manifolds} 
Mirror symmetry for Fano manifolds (see Givental 
\cite{Givental:ICM}) gives oscillatory integral representations of 
solutions (e.g.~the $J$-function) 
to the quantum differential equation.  
Such integral representations are useful for the asymptotic analysis 
and for proving the Gamma conjecture.  
We discuss a mirror symmetry version of the Gamma conjecture 
(as studied in \cite{Iritani:Int, Iritani:periods, GI, AGIS, Iritani:toric_osc}) 
which would lead to a (partial) proof of the original (non-mirror-symmetric) Gamma conjecture. 

\para 
\label{para:MS_Gamma} 
Suppose that an $n$-dimensional Fano manifold $F$ is mirror 
to a Laurent polynomial 
$W=W(x_1,\dots,x_n)$ in $n$-variables. 
Conjecturally, we obtain $W$ by counting holomorphic discs 
with boundary on a Lagrangian torus in $F$ (see Remark \ref{rem:slag} and \cite{Hori-Vafa, Cho-Oh, Auroux, FOOO:toricI, NNU, CLLT, Tonkonog, GM, YS-Lin}). 
We assume that 
\begin{enumerate} 
\item the Newton polytope of $W$ contains the origin in its 
interior; and 
\label{W1}
\item all non-vanishing coefficients of $W$ are positive real. 
\label{W2}
\end{enumerate} 
Under these assumptions \eqref{W1}-\eqref{W2}, 
$W|_{(\R_{>0})^n}$ is a strictly convex function 
with respect to the coordinates $(\log x_1,\dots,\log x_n)$ and attains a 
global minimum at a unique point. 
\emph{Mirror symmetric Gamma conjecture} is the equality 
\begin{equation} 
\label{eq:ms-gamma}
\int_{(\R_{>0})^n} e^{-W/z} \frac{dx_1\cdots dx_n}{x_1\cdots x_n} 
= 
\int_F (z^{c_1} z^{\deg/2} J_F(0,-z)) \cup \hGamma_F 
\end{equation} 
for $z>0$, where $z^{c_1} = e^{c_1 \log z}$ and `$\deg$' 
is regarded as an endomorphism of $H^*(F)$ 
that sends $\phi_i$ to $\deg(\phi_i) \phi_i$. 
Note that the above assumptions ensure the convergence of 
the integral in the left-hand side, and that the left-hand side 
has the asymptotics $\sim e^{-T/z}$ as $z\to +0$, where 
$T =\min_{x\in (\R_{>0})^n} W(x)$. 
In this equation the parameter $\tau\in H^2(F)$ in the $J$-function 
is  zero. More generally, we conjecture that there is a family 
of Laurent polynomials $W_\tau$ 
depending on $\tau\in H^2(F)$ such that 
\begin{equation} 
\label{eq:ms-gamma_tau} 
\int_{(\R_{>0})^n} e^{-W_\tau/z} \frac{dx_1\cdots dx_n}{x_1\cdots x_n} 
= 
\int_F (z^{c_1} z^{\deg/2} J_F(\tau,-z)) \cup \hGamma_F   
\end{equation} 
whenever $\tau$ is a real class. Here we assume that 
$W=W_\tau$ satisfies \eqref{W1}-\eqref{W2} for all real 
$\tau \in H^*(F,\R)$. 
These formulae are known to be true for Fano toric manifolds or 
Fano complete intersections in them \cite{Iritani:Int,Iritani:periods} and 
the original Gamma conjecture for those spaces 
follows\footnote{To be more precise, the proof of the Gamma 
conjecture in \cite{GI} is conditional on the assumption that 
these spaces satisfy a mirror analogue of Conjecture $\cO$.} 
from such formulae, see \cite[Theorem 6.3, Theorem 7.5]{GI}. 

\begin{remark} 
As remarked above, mirror symmetry predicts that the oscillatory integrals  
and the $J$-function satisfy the same differential equation. 
Therefore, the essential content 
of the Gamma conjecture \eqref{eq:ms-gamma_tau} 
is contained in the asymptotic behaviour: 
\[
\int_{(\R_{>0})^n} e^{-W_\tau/z} \frac{dx_1\cdots dx_n}{x_1\cdots x_n} 
\sim \int_F z^{c_1} e^{-\tau} \cup \hGamma_F \qquad 
\text{as $z\to \infty$.} 
\]
\end{remark} 

\begin{remark} 
We can consider more general cycles and integrands in this conjecture. 
When a (typically non-compact) integration cycle 
$C_V \subset (\C^\times)^n$ 
is mirror to a coherent sheaf $V$ on $F$ (in the sense of homological mirror 
symmetry) and an $n$-form 
$\check{\varphi}_\tau e^{-W_\tau/z} 
\frac{dx_1\cdots dx_n}{x_1\cdots x_n}$ (with $\check{\varphi}_\tau$ 
a Laurent polynomial 
in $x_1,\dots,x_n$) is mirror to a class $\varphi\in H^*(F)$, 
we should have 
\[
\int_{C_V} \check{\varphi}_\tau  
e^{-W_\tau/z}  \frac{dx_1\cdots dx_n}{x_1\cdots x_n} 
= \int_F (z^{c_1} z^{\deg/2} \JJ_F(\tau,-z) \varphi)  \cup 
\hGamma_F \Ch(V) 
\] 
where $\JJ_F(\tau,z) \in \End(H^*(F))$ is a matrix extension of the $J$-function (which gives a fundamental solution of the quantum connection) 
\[
\JJ_F(\tau,z)\varphi = e^{\tau/z} \left ( 
\varphi + \sum_{d\in H_2(F,\Z), d\neq 0}  e^{\tau\cdot d} \sum_{j=0}^N 
\corr{\varphi, \frac{\phi^j}{z-\psi}}_{0,2,d} \phi_j \right)  
\]
and 
$\Ch(V) := (2\pi\iu)^{\deg/2} \ch(V) 
= \sum_{k\ge 0} (2\pi\iu)^k \ch_k(V)$  
is a modified Chern character. 
The formulae \eqref{eq:ms-gamma}-\eqref{eq:ms-gamma_tau} are special cases 
of this equality when 
$\varphi=1$ and $V = \cO$ because $\JJ_F(\tau,z)1 = J_F(\tau,z)$. 

The above equality with $C_V = (S^1)^n$, $V = \cO_{\rm pt}$ and $\varphi=\check{\varphi}_\tau=1$ has been proved by Tonkonog \cite{Tonkonog} for the potential function $W$ defined in terms of disc counting (see Remark \ref{rem:slag}). 
\end{remark} 

\section{Gamma conjecture for Calabi-Yau manifolds} 
The original Gamma conjecture for Fano manifolds was formulated without 
a reference to mirrors, but Gamma conjecture for 
Calabi-Yau manifolds (we treat in this paper) involves mirror symmetry. 
The Gamma conjecture 
in the Calabi-Yau case roughly speaking says that \emph{the vanishing 
cycle of the mirror family (at the singularity nearest the large complex 
structure limit) corresponds to the Gamma class 
on the symplectic side}. This would explain a motivic origin of 
the $\zeta$-values appearing in the Gamma class. 
Historically, the Gamma classes (or the zeta values) have been observed in the asymptotics of periods of Calabi-Yau manifolds 
since the beginning of mirror symmetry and closely related observations have been made 
by many people \cite{CdOGP, HKTY, Libgober, Horja, vEvS,Hosono,Golyshev:deresonating}.  

\para A version of the Gamma conjecture for Calabi-Yau manifolds 
presented in \cite[Conjecture A]{AGIS}, which originates 
from Hosono's conjecture \cite[Conjecture 2.2]{Hosono}, 
says the following. Let $Y$ be a Calabi-Yau manifold
and let $Z_\tau$ be a family of mirror 
Calabi-Yau manifolds parametrized by 
$\tau \in H^2(Y)$ equipped with a 
holomorphic volume form $\Omega_\tau$. 
When $\Omega_\tau$ is suitably normalized and 
a cycle $C_\tau\subset Z_\tau$ is mirror to a coherent sheaf $V$ on $Y$ 
in the sense of homological mirror symmetry, we expect the following 
asymptotics of periods: 
\[
\int_{C_\tau \subset Z_\tau} \Omega_\tau 
\sim \int_Y e^{-\tau} \hGamma_Y  
\Ch(V) 
\]
as $\Re (\tau \cdot d) \to -\infty$ for all non-zero effective curve classes 
$d \in H_2(Y,\Z)$, where $\Ch(V) = (2\pi\iu)^{\deg/2} \ch(V)$. 
The limit $\Re(\tau \cdot d) \to -\infty$ is known as the \emph{large radius limit}. 
When $V$ is the structure sheaf $\cO_Y$, we expect that 
$C_\tau$ is a vanishing cycle at a singularity 
``closest'' to the large complex structure limit\footnote
{However, the notion of the ``closeness'' here is unclear: 
in many examples, we can choose the closest singularity 
with respect to a given coordinate.}.  

\para The main goal 
of this paper is to derive the Calabi-Yau Gamma conjecture from 
the Fano case. 
For a Fano manifold $F$, we consider a Calabi-Yau 
manifold given either as the total space $K_F$ of the canonical bundle or 
as a smooth anticanonical section $Y\subset F$. 
It is expected that the mirror of $K_F$ 
is given by relative periods of the pair 
$((\C^\times)^n, W_\tau^{-1}(-1))$ and 
that the mirror of $Y$ is given by periods of 
(a compactification of) $W_\tau^{-1}(1)$. 

The mirror symmetric Gamma conjecture for local Calabi-Yau manifolds $K_F$ 
(with $F$ Fano toric manifolds) 
was recently studied by Wang \cite{Wang:local} using the Gross-Siebert mirrors. 
Wang considered the case where the cycle $C_\tau$ is mirror to the structure sheaf of a curve. 
The Gamma conjecture in the local case is also 
closely related to the \emph{Ap\'ery extensions} studied by 
Golyshev, Kerr and Sasaki \cite{GKS:Apery_ext}.

\para
Let $I_{K_F}$ and $I_Y$ denote the $I$-functions 
\cite{Coates-Givental} of $K_F$ and $Y$ respectively. 
\begin{align} 
\label{eq:I_K_F} 
I_{K_F}(\tau,z) &=e^{\tau/z} \sum_{d\in H_2(F,\Z)} e^{\tau\cdot d}
J_d(z) \prod_{k=0}^{c_1\cdot d-1} (-c_1 - kz) 
 \\ 
 \label{eq:I_Y} 
I_Y(\tau,z) & =  e^{\tau/z} \sum_{d\in H_2(F,\Z)} e^{\tau\cdot d} 
J_d(z) \prod_{k=1}^{c_1\cdot d} (c_1 + kz) 
\end{align} 
where $\tau \in H^2(F)$ and we expand $J_F(\tau,z) 
= e^{\tau/z} \sum_{d\in H_2(F,\Z)} e^{\tau\cdot d} J_d(z)$. 
We regard $I_{K_F}$ as an $H^*(K_F)$-valued function 
and $I_Y$ as an $H^*(Y)$-valued function. 
The quantum Lefschetz theorem of Coates-Givental \cite{Coates-Givental} 
implies that these $I$-functions are related to 
the $J$-functions of $K_F$ and $Y$ via a change of variables (mirror transformation), 
and therefore compute certain Gromov-Witten invariants of $K_F$ or $Y$. 

\begin{theorem} 
\label{thm:main} 
Let $\tau \in H^2(F,\R)$ and 
suppose that we have a Laurent polynomial $W_\tau$ satisfying 
\eqref{W1}-\eqref{W2} and the mirror symmetric Gamma 
conjecture \eqref{eq:ms-gamma_tau} at $\tau$. 
Let $s>0$ be a sufficiently large positive number. 
Then we have 
\begin{equation}
\label{eq:local-charge}  
(2\pi \iu) \int_{B_{\tau,s}} 
\frac{dx_1\cdots dx_n}{x_1\cdots x_n}  = 
\int_{K_F} I_{K_F}(\tau- c_1 \log (-s),-1) 
\cup \hGamma_{K_F}  \Ch(\cO_F) 
\end{equation} 
where $B_{\tau,s} 
= (\R_{>0})^n \cap \{W_\tau(x) \le s\}$ is a relative $n$-cycle with 
boundary in $W_\tau^{-1}(s)$ and we choose the branch $\Im(\log(-s)) = -\pi$. 
We also have 
\begin{equation} 
\label{eq:anticanonical-charge} 
s \int_{C_{\tau,s}} \frac{d\log x_1 \cdots d\log x_n}{dW_\tau} 
 = \int_Y I_Y(\tau - c_1 \log s,-1) \cup \hGamma_Y 
\end{equation} 
where $C_{\tau,s} = (\R_{>0})^n \cap W_\tau^{-1}(s)$ is an 
$(n-1)$-cycle in $W_\tau^{-1}(s)$. 
\end{theorem} 

\begin{remark} 
$C_{\tau,s}$ is the vanishing cycle at $s = T_\tau
:= \min_{x\in (\R_{>0})^n} W_\tau(x)$. 
In this theorem, $s$ needs to be (at least) greater than $T_\tau$ 
(otherwise the cycles $B_{\tau,s}, C_{\tau,s}$ are empty). 
\end{remark} 

\begin{remark} 
$\cO_F$ means the structure sheaf of the zero-section $F \subset K_F$. 
The ($2\pi\iu$-modified) Chern character $\Ch(\cO_F)$ lies in the compactly supported 
cohomology of $K_F$, and hence the integral in the right-hand side 
of \eqref{eq:local-charge} makes sense. 
\end{remark} 

\begin{remark} 
On the right-hand side of 
\eqref{eq:local-charge}-\eqref{eq:anticanonical-charge}, 
the argument $z$ in the $I$-function is specialized 
to $-1$. Note however that $z^{c_1(TX)} 
z^{\deg/2} I_X(\tau,-z) = I_X(\tau,-1)$ for $X=K_F$ or $Y$.  
\end{remark} 

\begin{corollary} 
Suppose that we have a family 
$\{W_\tau\}_{\tau\in H^2(F,\R)}$ 
of Laurent polynomial mirrors satisfying \eqref{W1}-\eqref{W2} and 
the mirror symmetric Gamma conjecture \eqref{eq:ms-gamma_tau}.  
Let $B_{\tau,-1}$ denote the parallel translate of $B_{\tau,s}$ 
along an anti-clockwise path in the $s$-plane (increasing $\arg s$ by $\pi$). 
It is a relative cycle with boundary in $\{W_\tau(x)+1=0\}$.  
Then we have the asymptotics 
\begin{align*}
(2\pi \iu) \int_{B_{\tau,-1}} \frac{dx_1\cdots dx_n}{x_1\cdots x_n} 
& \sim \int_{K_F} e^{-\tau} \hGamma_{K_F} \Ch(\cO_F) \\ 
 \int_{C_{\tau,1}} \frac{d \log x_1\cdots d\log x_n}{d W_\tau} 
& \sim \int_Y e^{-\tau} \hGamma_Y 
\end{align*} 
as $\Re(\tau\cdot d) \to -\infty$ for all non-zero effective curve classes 
$d\in H_2(F,\Z)$. 
\end{corollary} 

\begin{example} 
\label{exa:toric} 
As we mentioned earlier, the mirror symmetric Gamma conjecture 
holds for a Fano toric variety. 
Let $F$ be a Fano toric variety and let 
$D_1,\dots,D_c$ be all prime toric divisors of $F$. 
We write $b_j \in \Z^n$ for the primitive vector of the 1-dimensional 
cone of the fan for $F$ corresponding to $D_j$. 
The $J$-function of $F$ is given by \cite{Givental:toric} 
\[
J_F(\tau,z) = e^{\tau/z} \sum_{d\in H_2(F,\Z)} e^{\tau\cdot d} 
\prod_{j=1}^c \frac{\prod_{k=-\infty}^0 D_j + kz}{ 
\prod_{k=-\infty}^{D_j\cdot d} D_j+kz}. 
\]
The Laurent polynomial $W(x) = \sum_{j=1}^c 
e^{-\lambda_j} x^{b_j}$ satisfies the mirror symmetric Gamma conjecture 
\eqref{eq:ms-gamma_tau} at $\tau = -\sum_{j=1}^c \lambda_j D_j$ 
\cite[Theorem 4.14]{Iritani:Int}. 
By applying the first half of Theorem \ref{thm:main} to $F$, 
we obtain (when $\sum_{j=1}^c \lambda_j D_j$ is sufficiently ample) 
\[
\vol \left\{ t\in\R^n : \sum_{j=1}^c e^{-\lambda_j + \corr{t, b_j}} \le 1  
\right\} 
= \int_{F} \sum_{d\in H_2(F,\Z)} 
e^{\sum_{j=1}^c \lambda_j (D_j - D_j \cdot d)}
\frac{\prod_{j=1}^c D_j \Gamma(D_j - D_j\cdot d)}
{\Gamma(c_1+1-c_1\cdot d)} 
\] 
where $\vol$ stands for the Euclidean volume 
and $c_1 = D_1+\cdots +D_c$. 
The leading asymptotics at the large-radius limit $\sum_{j=1}^c \lambda_j D_j\to \infty$ gives the classical  Duistermaat-Heckman formula 
\[
\vol(P_\lambda) = \int_{F} e^{\sum_{j=1}^c \lambda_j D_j} 
\]
where $P_\lambda = \{ t \in \R^n : \corr{t,b_j} \le \lambda_j, j=1,\dots,c\}$ 
is the moment polytope of $F$ with respect to a symplectic 
form in the class $\sum_{j=1}^c \lambda_j D_j$. 
\end{example} 

\section{Laplace transformation and a proof of Theorem \ref{thm:main}} 
\label{sec:Lap_proof}
We assume the mirror symmetric Gamma conjecture \eqref{eq:ms-gamma_tau}. 
Let $\cI(z)$ denote 
the quantity in \eqref{eq:ms-gamma_tau} 
and set $t = 1/z$: 
\begin{align} 
\label{eq:oscint} 
\begin{split}
\cI(1/t) 
& = \int_F (t^{-c_1} t^{-\deg/2} J_F(\tau,-1/t) )\cup \hGamma_F \\ 
& = \int_{(\R_{>0})^n} e^{ - t W_\tau } \frac{dx_1\cdots dx_n}
{x_1\cdots x_n}.
\end{split} 
\end{align} 
We calculate the Laplace transform of $\cI(1/t)$  
\[
\hcI(s) = 
\int_0^\infty e^{st}  \cI(1/t) dt 
\]
for negative real $s<0$. 
This integral is convergent because $\cI(t)$ 
is $O((\log t)^n)$ as $t\to 0$ 
by the first formula of \eqref{eq:oscint} 
(see also \eqref{eq:J_degrade} below) and 
$\cI(t)$ has exponential decay as $t\to +\infty$ by 
the second formula of \eqref{eq:oscint}. 
Using the first formula of \eqref{eq:oscint}, we obtain the following. 

\begin{proposition}[cf.~{\cite[Proposition 5.1]{Iritani:periods}}]
\label{prop:hcI_series} 
For a sufficiently negative $s\ll 0$, we have 
\[
\hcI(s) = \frac{1}{-s}\int_{F} \tI_Y(\tau - c_1 \log s,-1) \cup 
e^{-\pi \iu c_1} \Gamma(1-c_1)  \hGamma_F
\]
where $\tI_Y(\tau,z)$ is defined by the same formula \eqref{eq:I_Y} 
as $I_Y(\tau,z)$ 
but is regarded as a function taking values in $H^*(F)$ $($instead of in $H^*(Y))$ 
and we choose the branch $\Im \log s = \pi$ 
\end{proposition}
\begin{proof} 
We expand $J_F(\tau,z) = 
e^{\tau/z} \sum_{d\in H_2(F,\Z)}e^{\tau\cdot d} J_d(z) $. 
The homogeneity of Gromov-Witten invariants implies that 
$J_d(z) = \sum_i J_{d,i}  z^{-c_1\cdot d-\deg \phi_i/2}\phi_i$ 
for some $J_{d,i} \in \C$. This shows: 
\begin{equation} 
\label{eq:J_degrade}
t^{-c_1} t^{-\deg/2} J_F(\tau,-1/t) = 
e^{-\tau}\sum_{d\in H_2(F,\Z)} e^{\tau\cdot d}
J_d(-1)   t^{-c_1 + c_1\cdot d}. 
\end{equation} 
We compute the Laplace transform termwise, using the formula 
\begin{align*} 
\int_0^\infty e^{st} t^{-c_1+c_1 \cdot d} dt 
 & = (-s)^{c_1-c_1\cdot d-1} \Gamma(1+c_1\cdot d-c_1) \\ 
 & = \frac{1}{-s} (-s)^{c_1} s^{-c_1\cdot d} 
 \Gamma(1-c_1) \prod_{k=1}^{c_1\cdot d}(k-c_1). 
\end{align*} 
Then we arrive at the right-hand side of the proposition. 
It suffices to show that we can interchange the sum and the integral 
for $s\ll 0$. 
The convergence of the small quantum product of the Fano manifold 
$F$ implies the estimate (see \cite[Lemma 4.1]{Iritani:coLef}) 
\begin{equation}
\label{eq:estimate_J} 
|J_{d,i}| \le C_1 \frac{C_2^{c_1\cdot d}}{(c_1\cdot d)!}
\end{equation} 
for some $C_1,C_2>0$. 
Hence a partial sum has the estimate 
\[
\left| e^{-\tau} 
\sum_{d\in H_2(F,\Z), c_1\cdot d \le N} e^{\tau \cdot d} 
J_d(-1) t^{-c_1+c_1 \cdot d} \right| 
\le C_3(\tau) e^{C_4(\tau) t} \quad (\forall N>0)
\]
for some $C_3(\tau), C_4(\tau)>0$ 
and a norm $|\cdot|$ on $H^*(F)$. 
Therefore the sum and the integral can be interchanged 
when $s<-C_4(\tau)$ by the dominated convergence theorem. 
\end{proof} 

Next using the second formula of \eqref{eq:oscint}, we obtain the following. 
\begin{proposition}[cf.~proof of {\cite[Proposition 5.9]{Iritani:periods}}]
\label{prop:Hilbert_trans} 
For $s<0$, we have 
\[
\hcI(s) = \int_{(\R_{>0})^n} 
\frac{1}{W_\tau(x)-s} \frac{dx_1 \cdots dx_n}{x_1\cdots x_n}
\]
Moreover, the right-hand side makes sense as an analytic function of  
$s\in \C\setminus [T_\tau,\infty)$ where $T_\tau = \min_{x\in (\R_{>0})^n} 
(W_\tau(x)) >0$.  
\end{proposition} 
\begin{proof} 
We have 
\[
\hcI(s) = \int_0^\infty e^{st} dt \int_{(\R_{>0})^n} e^{-tW_\tau} 
\frac{dx_1 \cdots dx_n}{x_1\cdots x_n}
\]
Since the integrand $e^{st} e^{-t W_\tau}$ is positive, 
we can use the Fubini theorem and perform the integration in $t$ first. 
This gives the expression in the proposition. 
The integral on the right-hand side converges for all $s \in \C \setminus [T_\tau,\infty)$ 
because of the estimate 
\begin{equation} 
\label{eq:bound_uniform_in_s}
\left|\frac{1}{W_\tau(x)-s} \right| \le 
\left(\sup_{u\ge T_\tau} \left| \frac{u+1}{u-s}\right| \right) \frac{1}{W_\tau(x)+1} 
\qquad x\in (\R_{>0})^n 
\end{equation}
and the fact that it converges at $s=-1$. 
\end{proof} 

\begin{remark} When we choose $F$ to be a degree $(1,\dots,1)$ hypersurface in $(\PP^1)^{l+1}$, 
the integral expression for $\hcI(s)$ in Proposition \ref{prop:Hilbert_trans} gives the $l$-loop Banana Feynman amplitude \cite{Vanhove:exp,BFKNS,Iritani:Feynman}. 
\end{remark} 

\begin{proof}[Proof of Theorem \ref{thm:main}] 
Choose a sufficiently large positive number $M\gg 0$.  
We consider the integral of $\hcI(\tau,s)$ in $s$ along a path 
which starts from $s=-M$ and is contained 
in the region $\C \setminus [T_\tau,\infty)$. 
\[
\cM(s) := \int_{-M}^s \hcI(s') ds'  
\]
Using the formula in Proposition \ref{prop:Hilbert_trans} and the Fubini theorem\footnote
{The Fubini theorem applies because $s\mapsto \int_{(\R_{>0})^n} \frac{1}{|W_\tau(x)-s|} 
\frac{dx_1\cdots dx_n}{x_1\cdots x_n}$ is continuous on 
$\C\setminus [T_\tau,\infty)$ (by the estimate \eqref{eq:bound_uniform_in_s}), 
and therefore integrable along any compact path in 
$\C\setminus [T_\tau,\infty)$. }, 
we integrate out $s'$ to find 
\[
\cM(s) = \int_{(\R_{>0})^n} \log\left( \frac{W_\tau(x)+M}{W_\tau(x)-s} \right)  
\frac{dx_1 \cdots dx_n}{x_1\cdots x_n}.  
\]
The function $\cM(s)$ is analytic for $s\in \C\setminus [T_\tau,\infty)$. 
We compute the `jump' of this function across the branch cut $[T_\tau,\infty)$. 
For $f(s) = \log((W_\tau(x)+M)/(W_\tau(x)-s))$, we have
\[
\lim_{\epsilon \to +0} 
(f(s+\iu \epsilon) - f(s-\iu \epsilon)) = \begin{cases} 
2\pi \iu & \text{if $W_\tau(x) < s$} \\
0 & \text{if $W_\tau(x) >s$}  
\end{cases}
\]
Moreover, $|f(s+\iu \epsilon)- f(s-\iu \epsilon)|$ 
is bounded by an integrable function independent 
of $\epsilon \in (0,1)$, i.e.~when $0<\epsilon<1$, 
\[
|f(s+\iu \epsilon) - f(s-\iu \epsilon)| 
= \left| \log\frac{W_\tau(x)-s+\iu \epsilon}{W_\tau(x) -s - \iu \epsilon}
\right| \le 
\begin{cases} 
2 \pi & \text{if $W_\tau(x)\le s+1$}  \\
2/(W_\tau(x)-s) &\text{if $W_\tau(x) \ge s+1$} 
\end{cases}
\]
The function on the right-hand side is bounded by a constant multiple 
of $1/(W_\tau(x)+1)$ on $(\R_{>0})^n$, and hence is integrable 
with respect to the measure $\frac{dx_1\cdots dx_n}{x_1\cdots x_n}$. 
Therefore the jump of $\cM(s)$ is given by 
\begin{equation} 
\label{eq:cM_jump} 
\lim_{\epsilon \to +0} 
(\cM(s+\iu \epsilon) - \cM(s-\iu \epsilon)) = 
2\pi \iu 
\int_{(\R_{>0})^n \cap \{W_\tau(x) \le s\} }  \frac{dx_1 \cdots dx_n}{x_1\cdots x_n}. 
\end{equation} 
for $s>T_\tau$. This equals the left-hand side of \eqref{eq:local-charge}. 
On the other hand, Proposition \ref{prop:hcI_series} gives a convergent 
power series expansion of $\hcI(s)$ around $s=\infty$. 
By integrating the power series termwise, we get for $|s|\gg 0$, 
\[
\cM(s) = - \int_F\sum_{d\in H_2(F,\Z)} e^{-\tau+\tau \cdot d} 
\tI_d(-1) \frac{s^{c_1-c_1 \cdot d} -(-M)^{c_1-c_1\cdot d}}
{c_1-c_1\cdot d} 
\cup e^{-\pi \iu c_1} \Gamma(1-c_1) \hGamma_F 
\]
where we write $\tI_Y(\tau,z) = e^{\tau/z} 
\sum_{d\in H_2(F,\Z)} e^{\tau \cdot d} \tI_d(z)$.  
Here the branch of $s^a = \exp(a \log s)$ is chosen 
so that $\Im \log s \in (0,2\pi)$. 
From this we compute, for $s\gg 0$, 
\begin{align} 
\nonumber 
&\lim_{\epsilon \to +0} 
(\cM(s+\iu \epsilon) - \cM(s-\iu \epsilon))  \\ 
\label{eq:cM_jump_series}
& = 
- \int_F \sum_{d\in H_2(F,\Z)} e^{-\tau+\tau \cdot d} \tI_d(-1) 
\frac{(1-e^{2\pi \iu c_1})s^{c_1 - c_1 \cdot d}}{c_1 - c_1 \cdot d} 
\cup e^{-\pi \iu c_1} \Gamma(1-c_1) \hGamma_F \\ 
\nonumber 
& = \int_F \sum_{d\in H_2(F,\Z)} 
e^{-\tau + \tau \cdot d} J_d(-1) \left( 
\prod_{k=0}^{c_1\cdot d-1}(c_1-k) \right) 
\frac{1-e^{2\pi \iu c_1}}{-c_1} s^{c_1-c_1\cdot d} e^{-\pi \iu c_1} 
\Gamma(1-c_1) \hGamma_F \\ 
\nonumber 
& = \int_F I_{K_F}(\tau-  c_1 \log (-s), -1) 
\cup \Gamma(1-c_1) \hGamma_F \frac{1-e^{2\pi \iu c_1}}{-c_1} 
\end{align} 
where we choose the branch $\Im \log(-s) = -\pi$. 
The last line equals the right-hand side of \eqref{eq:local-charge}. 
Comparing this with \eqref{eq:cM_jump}, we obtain \eqref{eq:local-charge}. 

The equality \eqref{eq:anticanonical-charge} can be obtained from \eqref{eq:local-charge} 
by differentiating in $s$. It is obvious that the left-hand side of \eqref{eq:anticanonical-charge} 
multiplied by $2\pi\iu/s$ 
is the derivative of the left-hand side of \eqref{eq:local-charge}
Using the expression \eqref{eq:cM_jump_series} for the right-hand side of 
\eqref{eq:local-charge}, we get that the derivative of the right-hand side 
of \eqref{eq:local-charge} is 
\[
s^{-1} \int_F \sum_{d\in H_2(F,\Z)} e^{-\tau + \tau \cdot d} \tI_d(-1) 
s^{c_1- c_1\cdot d} 
(e^{\pi \iu c_1} - e^{-\pi \iu c_1}) \Gamma(1-c_1) \hGamma_F. 
\]
Using $\hGamma_Y = \hGamma_F/\Gamma(1+c_1) = 
\hGamma_F \Gamma(1-c_1) \frac{e^{\pi \iu c_1}- e^{-\pi \iu c_1}}{2\pi \iu c_1}$, 
we see that this equals the left-hand side of \eqref{eq:anticanonical-charge} multiplied 
by $2\pi \iu /s$. 
\end{proof}

\section{Generalizations to a sum of line bundles/complete intersections}
The above method can be easily generalized to  
a sum of line bundles over a Fano manifold $F$ or complete intersections in $F$. 
Consider the decomposition (the so-called ``nef partition'') 
\[
c_1 = v_0 + v_1 + \cdots +  v_c 
\]
of the first Chern class $c_1 = c_1(TF)$ such that 
each $v_i\in H^2(F,\Z)$ is nef. Suppose that we have a line 
bundle $\cV_i \to F$ with $c_1(\cV_i) = v_i$ for $i=1,\dots,c$. 
In this section, we consider the Gamma conjecture for 
the total space $\cV^\vee = \cV_1^\vee 
\oplus \cdots \oplus \cV_c^\vee$ of the sum of the anti-nef line bundles 
$\cV_1^\vee,\dots,\cV_c^\vee$ or a smooth complete intersection 
$Y \subset F$ cut out by a transversal section of $\cV = 
\cV_1\oplus \cdots \oplus \cV_c$. Note that the first Chern classes 
of these spaces are given by the nef class $v_0$. When $v_0=0$, 
these spaces are numerically Calabi-Yau. 

\para 
\label{para:nef_partition_W}
Let $W^{(0)},\dots,W^{(c)}$ be Laurent polynomial whose coefficients 
are positive real. We consider the family of Laurent polynomials 
\begin{equation} 
\label{eq:W_r} 
W_r = W^{(0)}+ r_1 W^{(1)} + \cdots + r_c W^{(c)} 
\end{equation} 
parametrized by $r =(r_1,\dots,r_c) \in (\R_{>0})^c$ and assume that 
the origin is contained in the interior of the Newton polytope of $W_r$; 
then $W_r$ satisfies the assumptions \eqref{W1}-\eqref{W2}. 
Suppose that the mirror symmetric Gamma conjecture \eqref{eq:ms-gamma_tau} 
holds for $W_r$ via the identification of parameters 
\[
\tau = \tau(r) = \tau_0 + v_1 \log r_1 + \cdots + v_c \log r_c 
\]
for some $\tau_0\in H^2(F)$, that is, we have 
\begin{equation} 
\label{eq:ms-gamma_r}
\int_{(\R_{>0})^n} e^{-W_r/z} \frac{dx_1\cdots dx_n}{x_1\cdots x_n} 
= \int_F (z^{c_1} z^{\deg/2} J_F(\tau, -z)) \cup \hGamma_F  
\end{equation}  
with $\tau = \tau(r)$. 
Such a family of Laurent polynomials would arise as the potential 
function of holomorphic discs, see Remark \ref{rem:slag} below. 
We also expect that $W^{(0)} = 0$ when $v_0=0$, but we do not need 
to assume this in the following discussion.

\para 
In this situation, we expect the following mirror correspondence: 
\begin{align*} 
\cV^\vee & \longleftrightarrow 
\text{relative (exponential) periods of 
$\left((\C^\times)^n, (W^{(1)})^{-1}(-1/r_1) \cup \cdots \cup 
(W^{(c)})^{-1}(-1/r_c)\right)$} \\ 
Y & \longleftrightarrow 
\text{(exponential) periods of  $(W^{(1)})^{-1}(1/r_1) \cap \cdots 
\cap (W^{(c)})^{-1}(1/r_c)$} 
\end{align*}  
where `exponential periods' mean integrals of algebraic forms 
multiplied by $e^{-W^{(0)}/z}$; 
they are usual periods when $W^{(0)}=0$. 

\para Introduce the $I$-functions \cite{Coates-Givental} 
for $\cV^\vee$ and $Y$ as follows: 
\begin{align*} 
I_{\cV^\vee}(\tau,z) 
& = e^{\tau/z} \sum_{d\in H_2(F,\Z)} 
e^{\tau \cdot d} J_d(z) \prod_{i=1}^c \prod_{k=0}^{v_i \cdot d-1} (-v_i -kz) \\ 
I_Y(\tau,z) & =  e^{\tau/z} \sum_{d\in H_2(F,\Z)} 
e^{\tau \cdot d} J_d(z) \prod_{i=1}^c \prod_{k=1}^{v_i \cdot d} (v_i + kz) 
\end{align*} 
where the $J$-function of $F$ is expanded as 
$J_F(\tau,z) = e^{\tau/z} \sum_{d\in H_2(F,\Z)} e^{\tau \cdot d}J_d(z) $ as before. 
We regard $I_{\cV^\vee}(\tau,z)$ as $H^*(\cV^\vee)$-valued function 
and $I_Y(\tau,z)$ as $H^*(Y)$-valued function. 
These $I$-functions are related to the respective $J$-functions via 
mirror transformation and therefore compute certain Gromov-Witten 
invariants of $\cV^\vee$ or $Y$. 

\begin{theorem}
\label{thm:generalizations}  
Let $W_r$ be a family of Laurent polynomials as in \eqref{eq:W_r} satisfying 
the mirror symmetric Gamma conjecture \eqref{eq:ms-gamma_r}. 
Let $s_1,\dots,s_c>0$ be sufficiently large positive numbers. We have 
\[ 
(2\pi \iu)^c \int_{B_{s}} e^{-W^{(0)}/z} 
\frac{dx_1\cdots dx_n}{x_1\cdots x_n}  
= 
\int_{\cV^\vee} 
\left( z^{v_0} z^{\frac{\deg}{2}} 
I_{\cV^\vee}(\tau_0-v \log (-s) , -z) \right) 
\cup \hGamma_{\cV^\vee} \Ch(\cO_F) 
\] 
where $v\log (-s) = \sum_{i=1}^c v_i \log (-s_i)$, 
$B_{s}=(\R_{>0})^n \cap 
\bigcap_{i=1}^c \{W^{(i)}(x) \le s_i\}$ is a relative $n$-cycle
and we choose the branch $\Im(\log (-s_i)) = -\pi$.  
We also have 
\[
\left(\prod_{i=1}^c s_i \right)\int_{C_{s}} e^{-W^{(0)}/z} 
\frac{d \log x_1 \cdots d \log x_n}{d W^{(1)} \cdots 
d W^{(c)}} 
= \int_{Y} 
\left (z^{v_0} z^{\frac{\deg}{2}} 
I_Y(\tau_0- v\log s, -z) \right) 
\cup \hGamma_Y, 
\] 
where $v \log s = \sum_{i=1}^c v_i \log s_i$ 
and $C_{s}=
(\R_{>0})^n \cap \bigcap_{i=1}^c (W^{(i)})^{-1}(s_i)$ 
is an $(n-c)$-cycle. 
\end{theorem} 
\begin{remark} 
For a generic $s\in (\R_{>0})^c$, $C_{s}\subset 
(\R_{>0})^n$ is either empty or a 
smooth complete intersection of codimension $c$. 
Therefore the measure $d \log x_1 \cdots d \log x_n/(d W^{(1)} 
\cdots dW^{(c)})$ on $C_{s}$ makes sense.  
\end{remark} 

\begin{proof}[Proof of Theorem \ref{thm:generalizations}] 
The proof is parallel to that of Theorem \ref{thm:main}, 
so we only give an outline. 
We consider the Laplace transform of the quantity \eqref{eq:ms-gamma_r} 
with respect to $r_1,\dots,r_c$: 
\[
\hcI(s,z) := 
\int_{[0,\infty)^c} e^{\sum_{i=1}^c r_i s_i/z} 
\cI\left(\tau(r),z \right) 
dr_1\cdots dr_c 
\]
with $s_i<0$ and $z>0$,  
where $\cI(\tau(r),z)$ denotes the quantity in \eqref{eq:ms-gamma_r}. 
Using the left-hand side of \eqref{eq:ms-gamma_r} 
and integrating $r_i$ out, we find 
\begin{equation} 
\label{eq:Hilbert_trans_mult} 
\hcI(s,z) = z^c \int_{(\R_{>0})^n} 
\frac{e^{-W^{(0)}/z}}{
(W^{(1)} -s_1) \cdots (W^{(c)} -s_c)} 
\frac{dx_1\cdots dx_n}{x_1\cdots x_n}.  
\end{equation} 
On the other hand, by using the right-hand side of 
\eqref{eq:ms-gamma_r} and computing the Laplace transform 
termwise, we obtain 
\begin{align}
\label{eq:hcI_series_mult}  
\hcI(s,z) 
& = \frac{z^c}{\prod_{i=1}^c (-s_i)} 
\int_F \left( z^{v_0} z^{\deg/2} 
\tI_{Y}(\tau_0 - v \log s, z) \right) 
\cup \left(\prod_{i=1}^c e^{-\pi \iu v_i} \Gamma(1-v_i) \right) 
\hGamma_F 
\end{align} 
where $\tI_Y(\tau,z)$ is defined by the same formula as $I_Y(\tau,z)$ 
but is regarded as a function taking values in $H^*(F)$ and 
we choose the branch $\Im (\log s_i) = \pi$. 
The interchange of the sum and the integral can be justified 
when $s_j \ll 0$ 
by the estimate \eqref{eq:estimate_J} as before;
this simultaneously shows the convergence of the Laplace 
transformation for $s_i \ll 0$. 
The integral representation \eqref{eq:Hilbert_trans_mult} then shows 
that $\hcI(s,z)$ can be analytically continued to a holomorphic 
function for $s\in (\C \setminus [0,\infty))^c$. 

We choose a sufficiently large $M\gg 0$ and consider the integral 
\[
\cM(s,z) = \int_{-M}^{s_1} ds'_1 \int_{-M}^{s_2} ds'_2 
\cdots \int_{-M}^{s_c} ds'_c \, \hcI(s',z) 
\]
where the path of integration is contained in the region 
$(\C\setminus [0,\infty))^c$.  
Using \eqref{eq:Hilbert_trans_mult}, we find that 
\[
\cM(s,z) = z^c \int_{(\R_{>0})^n} 
\left( \prod_{i=1}^c \log \left( \frac{W^{(i)} +M}{W^{(i)} -s_i} 
\right) \right) e^{-W^{(0)}/z} \frac{dx_1\cdots dx_n}{x_1\cdots x_n} 
\]
which is analytic in $s \in (\C\setminus [0,\infty))^c$. 
We compute iteratively the jump of $\cM(\tau,s,z)$ across the branch cut 
$s_j \in [0,\infty)$ for $j=1,\dots,c$. Writing  
$\Delta_j f (s_1,\dots,s_c) = 
\lim_{\epsilon \to +0} (f(s_1,\dots,s_j+ \iu \epsilon, \dots, s_c) 
- f(s_1,\dots,s_j-\iu \epsilon, \dots, s_c))$ for a function $f$ 
defined on $(\C\setminus [0,\infty))^c$,  we have 
\[
(\Delta_1\cdots \Delta_c \cM)(s,z) = (2\pi \iu z)^c 
\int_{B_{s}} 
e^{-W^{(0)}/z} \frac{dx_1\cdots dx_n}{x_1\cdots x_n}.  
\] 
Calculating the same jump using the power series expansion 
\eqref{eq:hcI_series_mult}, we obtain the first formula of the 
proposition. The second formula follows from the first by differentiating 
it in $s_1,\dots,s_c$.
\end{proof} 

\begin{remark} 
\label{rem:slag} 
The Laurent polynomial mirror $W_r$ as in \eqref{eq:W_r} 
should arise from disc counting as follows \cite{Hori-Vafa, Cho-Oh, Auroux,FOOO:toricI, NNU, CLLT, Tonkonog, GM, YS-Lin}. 
Suppose that we have simple normal crossing divisors $D_0,D_1,\dots,D_c$ 
such that $-K_F = D_0 +D_1+ \cdots + D_c$ with $v_i = [D_i]$ 
and that we have a special Lagrangian torus fibration 
$F \setminus \bigcup_{i=0}^c D_i \to B$. 
We consider holomorphic discs of Maslov index two with boundaries on a 
Lagrangian torus fiber $L\subset F$. 
The potential function is given by counting such holomorphic discs: 
\[
W_r = \sum_{\substack{\beta \in H_2(F,L) \\ 
\mu(\beta)=2}} n_\beta x^{\partial \beta}  \prod_{i=0}^c r_i^{D_i\cdot \beta} 
\]
where $n_\beta$ is the open Gromov-Witten invariant, 
a virtual count of stable holomorphic discs with boundary 
in $L$ of class $\beta\in H_2(F,L)$. 
The factor $\prod_{i=0}^c r_i^{D_i\cdot \beta}$ can be interpreted as 
the exponentiated area $e^{\int_\beta \tau}$ of a disc 
with $\tau = \sum_{i=0}^c v_i \log r_i$.  
The Maslov index $\mu(\beta)$  
is given by $\mu(\beta) = 2\beta \cdot (D_0+\cdots +D_c)$ 
and hence $\mu(\beta) =2$ implies that there exists $i\in \{0,1,\dots,c\}$ 
such that $D_i \cdot \beta =1$ and $D_j \cdot \beta =0$ for all $j\neq i$. 
This gives a function of the form \eqref{eq:W_r}. 
\end{remark} 


\begin{example} 
Let $F$ be a Fano toric variety as in Example \ref{exa:toric}. 
With notation as there, suppose that we have a partition 
$\{D_1,\dots,D_c\} = U \sqcup V_1 \sqcup \cdots \sqcup V_l$ 
such that $u = \sum_{D_i\in U} D_i$ and $v_j = \sum_{D_i\in V_j} D_i$ 
($1\le j\le l$) 
are nef.   
Let $F'\subset F$ be a smooth Fano hypersurface 
in the class $u$. 
Applying the second half of Theorem \ref{thm:generalizations} 
to $\cV = \cO(u)$,  
we obtain the mirror symmetric Gamma conjecture for $F'$ 
\begin{equation*} 
\int_{Z \cap (\R_{>0})^n} e^{-W'/z} 
\frac{d \log x_1 \wedge \cdots \wedge  
d\log x_n}{d(\sum_{D_i\in U} e^{-\lambda_i} x^{b_i})} = 
\int_{F'} (z^{c_1(F')} z^{\deg/2} I_{F'}(\tau,-z)) \cup  \hGamma_{F'} 
\end{equation*} 
for $\tau = -\sum_{i=1}^c  \lambda_i D_i$, 
where we set 
$W' =e^{-\lambda_{k+1}} x^{b_{k+1}} + \cdots + e^{-\lambda_c} 
x^{b_c}$ 
and 
$Z=\{x\in (\C^\times)^n : e^{-\lambda_1} x^{b_1} + 
\cdots + e^{-\lambda_k} x^{b_k} = 1\}$. 
This result already appeared in \cite[Theorem 5.7]{Iritani:periods}. 

We further apply the first half of Theorem \ref{thm:generalizations} 
to $F'$ and the nef partition $c_1(F') = v_1+\cdots+v_l$.  
The function $W' \colon Z \to \C$ is not (apparently) a Laurent polynomial. 
In most cases, however, we can find a torus chart $(\C^\times)^{n-1}\subset Z$ 
such that $Z \cap (\R_{>0})^n$ corresponds to $(\R_{>0})^{n-1}$ 
and that $W$ is a Laurent polynomial on that chart, see e.g.~\cite[\S 2]{CKP}. 
Also, as mentioned before, $I_{F'}$ and the $J$-function $J_{F'}$ 
coincide after a change of variables, but they are not necessarily the same. 
Nevertheless, we can easily check that 
the argument in the proof of Theorem \ref{thm:generalizations} 
applies equally to this situation where the mirror is not necessarily 
a Laurent polynomial and the $J$-function is replaced with the $I$-function. 
We obtain the following equality: 
\[ 
\int_{B}
\frac{d\log x_1\cdots d\log x_n}
{d(\sum_{D_i\in U} e^{-\lambda_i} x^{b_i})} 
= 
\int_{F} \sum_{d\in H_2(F,\Z)} e^{-\tau+\tau\cdot d}
\frac{\prod_{i=1}^c D_i \Gamma(D_i- D_i\cdot d)}
{\Gamma(u-u\cdot d) \prod_{i=1}^l \Gamma(v_i+1-v_i \cdot d)} 
\] 
where $B=\{x \in Z \cap (\R_{>0})^n : 
\sum_{D_i\in V_j} e^{-\lambda_i} x^{b_i} 
\le 1, j=1,\dots,l\}$ and $\tau = -\sum_{i=1}^c \lambda_i D_i$. 
\end{example}

\paragraph{Acknowledgments.} 
I thank Atsushi Kanazawa, Albrecht Klemm and Junxiao Wang for helpful discussions. I also thank an anonymous referees for valuable comments. 
This research is supported by JSPS Kakenhi Grant Number 20K03582, 21H04994.

\providecommand{\arxiv}[1]{\href{http://arxiv.org/abs/#1}{arXiv:#1}}

Department of Mathematics, 
Graduate School of Science, Kyoto University, 
Kitashirakawa-Oiwake-cho, Sakyo-ku, Kyoto, 606-8502, Japan 

{\tt iritani@math.kyoto-u.ac.jp}

\end{document}